\documentclass[11pt,reqno]{amsart}

\usepackage{soul}
\usepackage{amsthm}
\usepackage{amssymb}
\usepackage{latexsym}
\usepackage{multicol}
\usepackage{verbatim,enumerate}
\usepackage{accents}
\usepackage[usenames]{color}
\usepackage{hyperref}
\usepackage{hyperref}
\usepackage{amsmath, amscd}
\usepackage{soul}
\usepackage{dsfont}
\usepackage{tikz}
\usetikzlibrary{matrix,arrows,decorations.pathmorphing}
\usetikzlibrary{positioning}

\advance\textwidth by 1.2in \advance\oddsidemargin by -.6in \advance\evensidemargin by -.6in

\theoremstyle{plain}
\newtheorem{prop}{Proposition}
\newtheorem{thm}{Theorem}
\newtheorem{lem}{Lemma}
\newtheorem{cor}{Corollary}

\theoremstyle{definition}
\newtheorem{example}{Example}
\newtheorem{defn}{Definition}

\theoremstyle{remark}
\newtheorem{rem}{Remark}

\newcommand{\lie}[1]{\mathfrak{#1}}

\newcommand\bc{\mathbb C}
\newcommand\bn{\mathbb N}
\newcommand\bz{\mathbb Z}

\newcommand{\sign}{\operatorname{sgn}}
\newcommand{\ch}{\operatorname{ch}}

\def\a{\alpha}

\newenvironment{pf}{\proof}{\endproof}
\newcounter{cnt}
 \makeatletter
\def\mydggeometry{\makeatletter\dg@YGRID=1\dg@XGRID=20\unitlength=0.003pt\makeatother}
\makeatother \theoremstyle{remark}

\numberwithin{equation}{section}
\makeatletter
\def\section{\def\@secnumfont{\mdseries}\@startsection{section}{1}%
  \z@{.7\linespacing\@plus\linespacing}{.5\linespacing}%
  {\normalfont\scshape\centering}}
\def\subsection{\def\@secnumfont{\bfseries}\@startsection{subsection}{2}%
  {\parindent}{.5\linespacing\@plus.7\linespacing}{-.5em}%
  {\normalfont\bfseries}}
\makeatother

\begin{document}


\title{On tensor products of irreducible integrable representations}

\author{Shifra Reif}
\thanks{SR was partly funded by ISF Grant  No. 1221/17.}
\address{Bar-Ilan University,  Ramat Gan, Israel}
\email{shifra.reif@biu.ac.il}

\author{R. Venkatesh}
\thanks{RV was partially funded by the grants DST/INSPIRE/04/2016/000848 and MTR/2017/000347.}
\address{Indian Institute of Science, Bangalore, India}
\email{rvenkat@iisc.ac.in}

\begin{abstract}
We consider integrable category $\mathcal{O}$ representations of Borcherds--Kac--Moody algebras whose Cartan matrix is finite dimensional, and 
determine the necessary and sufficient conditions for which the tensor
product of irreducible representations from this category is isomorphic to another. 
This result generalizes a fundamental result of C. S. Rajan on unique factorization of tensor products of finite dimensional 
irreducible representations of finite dimensional simple Lie algebras over complex numbers. 
\end{abstract}
\maketitle
\section{Introduction}

A theorem of Rajan (see \cite{Rajan}) asserts the following fundamental property of tensor products of irreducible representations of a finite dimensional
simple Lie algebra $\lie g$:
Given non-trivial finite--dimensional irreducible $\lie g$--modules $V_1,\ldots,V_r$ and $W_1,\ldots ,W_s$ such that
$$V_1\otimes\ldots\otimes V_r\cong W_1\otimes\ldots\otimes W_s \ \text{as $\lie g$--modules} $$ if and only if
 $r=s$ and the factors are pairwise isomorphic as $\lie g$--modules up to a permutation of indices. 
This statement is equivalent to the family of irreducible finite--dimensional representations having unique factorization property in the 
Grothendieck ring of finite dimensional representations of $\lie g$.
Rajan proved his result by an inductive analysis of the characters of
tensor products, by fixing one of the variables, and passing to a suitable lower rank Lie algebra. 
A more direct and simpler proof of Rajan's theorem is obtained by the second author and Viswanath in \cite{VV12} and they also obtained a natural generalization of Rajan's theorem
to Kac--Moody algebras setting.
A natural category of representations to consider for Kac--Moody algebras is the category $\mathcal{O}^{\mathrm{int}}$, whose objects are integrable $\lie g$--modules in category $\mathcal{O}$, since 
both proofs of \cite{Rajan, VV12} heavily uses the Kac--Weyl character formula.

For a Kac--Moody algebra, the tensor product of two irreducible integrable representations can be irreducible when the Dynkin diagram
is not connected. So the unique factorization property can not hold for Kac--Moody algebras which are not indecomposable.
Even in this case, the unique factorization of tensor products fails in general for Kac--Moody algebras as there may exist 
non-trivial one dimensional representations in $\mathcal{O}^{\mathrm{int}}$.
However these are the only obstruction we have, i.e. uniqueness still holds up to twisting by one-dimensional representations for an indecomposable Kac--Moody
algebra $\lie g$. More
precisely, given non-trivial irreducible integrable $\lie g$-modules  $V_1,\ldots, V_r$ and $W_1,\ldots,W_s$ in category $\mathcal{O}$ such that
$$V_1\otimes\ldots\otimes V_r\cong W_1\otimes\ldots\otimes W_s$$ if and only if
 $r=s$ and there exists a permutation $\sigma$ of $\{1,\ldots,r\}$ such that $V_i\cong W_{\sigma(i)}\otimes Z_i \ \text{as $\lie g$--modules}$ 
for some one-dimensional $\lie g$-modules $Z_i$, $i=1,\ldots,r$.
Note that the one dimensional representations are precisely the units of the Grothendieck ring of the category $\mathcal{O}^{\mathrm{int}}$. 
Thus the unique factorization property still holds in Kac--Moody case up to reordering and multiplying by units.

So, we can naturally ask whether such a unique factorization of tensor products theorem holds for the irreducible integrable representations
of {\em{Borcherds--Kac--Moody algebras}} as they also admit
a character formula similar to the Kac--Weyl character formula.
However there are more obstacles in this case since the building blocks of
Borcherds--Kac--Moody algebras involve Heisenberg algebras. For example let $\lie g$ be the 
Borcherds--Kac--Moody algebra with only one imaginary simple root and no real simple roots and let $L(\lambda)$ be the irreducible integrable representation of $\lie g$
corresponding to the dominant weight $\lambda$.
Then it is easy to see that
$L(\lambda_1)\otimes \cdots \otimes L(\lambda_r)\cong  L(\mu_1)\otimes \cdots \otimes L(\mu_r)$
if and only if $\sum_{i=1}^r\lambda_i=\sum_{j=1}^r\mu_j$. 
So, one can not expect to get the unique factorization property for tensor products of irreducible integrable representations for general Borcherds--Kac--Moody algebras
as in the case of Kac--Moody algebras, even up to one dimensional twists. It is not hard to produce more counter examples when the Borcherds--Kac--Moody algebras has no real simple roots,
see the Section \ref{examplesuniquefails} for some more examples.

It is therefore natural to study when two tensor products (not necessarily same number of components)
of irreducible integrable representations of Borcherds--Kac--Moody algebras will be isomorphic to another. 
In this paper, we address this problem in full generality for 
 Borcherds--Kac--Moody
algebras whose Borcherds--Cartan matrix is finite.
We also prove we do get unique factorization of tensor products under some mild assumption on the integrable irreducible representations of Borcherds--Kac--Moody algebras
(see Section \ref{specialdominant} and Corollary \ref{uniquecor} for more details). 
\begin{thm}\label{introthm}
Let $\lie g$ be an indecomposable Borcherds--Kac--Moody algebra and let $r, s\in\mathbb{N}$ and let
$V_i, \ 1\le i\le r$ and $W_j, \ 1\le j\le s$ be special irreducible $\lie g$--modules in category $\mathcal{O}^{\rm{int}}$ such that
\begin{equation*}
V_1 \otimes \cdots \otimes V_r \cong W_1 \otimes \cdots \otimes W_s.
\end{equation*}
Then we have $r=s$ and
there exists a permutation $\sigma$ on $\{1, \cdots, r\}$ and one--dimensional $\lie g$--modules $Z_i$, $1\le i\le r$ such that 
$V_i=W_{\sigma(i)}\otimes Z_i.$
\end{thm}
\noindent
This theorem is a corollary of Theorem \ref{tensormainthm} which describes when two tensor products of integrable irreducible $\lie g$--modules are isomorphic to each other.

\vskip 2mm
We now explain the main strategy of our  proof. Let $\lambda_1,\ldots, \lambda_r, \ \mu_1,\ldots, \mu_s$ 
be dominant weights of $\lie g$ such that $V_i=L(\lambda_i)$ and $W_j=L(\mu_j)$ for all $1\le i\le r, \ 1\le j\le s$.
Since the category $\mathcal{O}^{\mathrm{int}}$ is completely reducible for any Borcherds--Kac--Moody algebra $\lie g$ (see \cite{Bor88}, \cite{kang}), the
isomorphism of the tensor products of $\lie g$--modules becomes equivalent to the equality of the corresponding characters.
So, taking formal characters on the both sides of tensor products one gets:
\begin{equation}\label{introchracter}
\ch L(\lambda_1)\cdots \ch L(\lambda_r)=\ch L(\mu_1) \cdots \ch L(\mu_s)
\end{equation}
Now by canceling out the Weyl denominators, simplifying we translate our main problem to the
equality of product of the normalized Weyl numerators,
\begin{equation}\label{introUlambda}
U_{\lambda_1}\cdots U_{\lambda_r}=U_{\mu_1}\cdots U_{\mu_s}
\end{equation} (see Section 2.3 for the precise definition of $U_\lambda$).
We show that both sides of the Equation (\ref{introUlambda}) can be further factorized to a product 
$$\prod_{i=1}^r \prod_{j=1}^{r_i} U_j^{\lambda_i}=\prod_{i=1}^s \prod_{j=1}^{s_i} U_j^{\mu_i}, $$
where the terms $U_j^{\lambda}$ are parametrized by the connected components of graphs associated to $\lambda$. We show that this factorization to such terms is unique. 
We prove this using the logarithm techniques developed in \cite{VV12}.
If the graphs of $\lambda_1,\ldots,\lambda_r,\mu_1,\ldots ,\mu_s$ are connected, then we get $r=s$ and $r_i=s_j=1$ for $i, j=1,\ldots,r$ 
and we obtain the unique factorization of the characters and using this we obtain the unique factorization of the tensor products of special integrable irreducible $\lie g-$modules.

In fact, we prove this unique factorization result for more general expressions than the product of normalized Weyl numerators.
We add a parameter $\chi$ to the normalized Weyl numerators where $\chi$ is  
a homomorphism on $W\times Q^{\mathrm{im}}_+$, where $W$ is the Weyl group of $\lie g$ and $Q^{\mathrm{im}}_+$ is the set of
non-negative integer linear combination of imaginary simple roots (see Section \ref{elmUlambchi} for more details). 
The unique factorization property is proved for any $\chi$. When $\chi$ is the sign character, 
we obtain the unique factorization properties of normalized Weyl numerators and hence the characters.

The paper is organized as follows: In Section \ref{section2}, we set up the notations and preliminaries. In
Section \ref{section3}, we introduce the elements $U(\lambda, \chi)$ generalizing the normalized Weyl numerators associated with a homomorphism $\chi$ and
prove the key properties concerning logarithm of $U(\lambda, \chi)$. In Section \ref{main theorem section}, we 
use these properties to prove our main results of the paper.


\vskip 2mm

\emph{\textbf{Acknowledgment.}{ The authors are thankful to Maria Gorelik for helpful conversations.}}


\section{Preliminaries}\label{section2}

\subsection{Borcherds--Kac--Moody algebra}
Throughout this paper our base field will be complex numbers, i.e., all the algebras and representations are complex--vector spaces. The complex numbers, integers, non-negative integers, 
and positive integers are denoted  by $\bc$, $\bz$, $\bz_+$, and $\bn$. 
We recall the definition of Borcherds--Kac--Moody algebras from \cite{Ej96}; they are also called generalized Kac--Moody algebras, see also \cite{Bor88, K90}. 
A real matrix  
$A=(a_{ij})_{i,j\in I}$ indexed by a finite set $I=\{1, \dots, n\}$  is said to be a \textit{Borcherds--Cartan\ matrix}
if the following conditions are satisfied for all $i,j\in I$:
\begin{enumerate}
\item $A$ is symmetrizable;
\item $a_{ii}=2$ or $a_{ii}\leq 0$;
\item $a_{ij}\leq 0$ if $i\neq j$ and $a_{ij}\in\mathbb{Z}$ if $a_{ii}=2$;
\item $a_{ij}=0$ if and only if $a_{ji}=0$.
\end{enumerate} 
Recall that a matrix $A$ is called symmetrizable if $DA$ is symmetric for some diagonal matrix $D=\mathrm{diag}(d_1, \ldots, d_n)$ with positive entries.
Denote by $I^{\mathrm{re}}=\{i\in I: a_{ii}=2\}$ and $I^{\mathrm{im}}=I\backslash I^{\mathrm{re}}$. 
The Borcherds--Kac--Moody algebra  $\lie g=\mathfrak{g}(A)$ associated to the Borcherds--Cartan matrix $A$ is the Lie algebra generated by 
$e_i, f_i, h_i$, $i\in I$ with the following defining relations:
\begin{enumerate}
 \item[(R1)] $[h_i, h_j]=0$ for $i,j\in I$
 \item[(R2)] $[h_i, e_k]=a_{i,k}e_i$,  $[h_i, f_k]=-a_{i,k}f_i$ for $i,k\in I$
 \item[(R3)] $[e_i, f_j]=\delta_{ij}h_i$ for $i, j\in I$
 \item[(R4)] $(\text{ad }e_i)^{1-a_{ij}}e_j=0$, $(\text{ad }f_i)^{1-a_{ij}}f_j=0$ if $i\in I^{\mathrm{re}}$ and $i\neq j$
 \item[(R5)] $[e_i, e_j]=0$ and $[f_i, f_j]=0$ if $a_{ij}=0$.
\end{enumerate}
\begin{rem} 
If $i\in I$ is such that $a_{ii}=0$, the subalgebra spanned by the elements $h_i,e_i,f_i$ is isomorphic to the three dimensional Heisenberg algebra and otherwise it is isomorphic to 
$\mathfrak{sl}_2$. \em{Note that we only consider  the Borcherds--Kac--Moody algebras associated with finite Borcherds--Cartan matrices in this paper.}
\end{rem}

\subsection{Root System}
In this subsection, we recall some of the basic properties of Borcherds--Kac--Moody algebras; see \cite{Ej96} for more details. 
The Borcherds--Kac--Moody algebra $\lie g$ is $\mathbb{Z}^{n}$--graded by defining
$$\text{
$\mathrm{deg}\ h_i=(0, \dots,0)$, $\mathrm{deg}\ e_i=(0,\dots,0,1,0,\dots,0)$ and 
$\text{deg}\ f_i=$ $(0,\dots,0,-1, 0,\dots,0)$}$$
where $\pm 1$ appears at the $i$--th position.
For a sequence $(k_1, \dots, k_n)$, we denote by $\lie g(k_1, \dots, k_n)$ the corresponding graded piece.
Let $\mathfrak{h}=\text{Span}_\mathbb{C}\{h_i: i\in I\}$ be the abelian subalgebra and let
$\mathfrak{E}=\text{Span}_{\mathbb{C}}\{D_i: i\in I\}$,
where $D_i$ denotes the derivation that acts on $\lie g(k_1, \dots, k_{n})$ by 
multiplication by the scalar $k_i$ and zero on the other graded components. Note that $D_i, i\in I$ are commuting derivations of $\lie g$.
The abelian subalgebra $\mathfrak{E}\ltimes \mathfrak{h}$ of $\mathfrak{E}\ltimes \mathfrak{g}$ acts by scalars on $\lie g(k_1,\dots, k_n)$
and giving a root space decomposition:
\begin{equation}\label{rootdec}\mathfrak{g}=\bigoplus _{\alpha \in (\mathfrak{E}\ltimes \mathfrak{h})^*}
\mathfrak{g}_{\alpha }, \ \mathrm{where} \ \mathfrak{g}_{\alpha }
:=\{ x\in \mathfrak{g}\ |\ [h, x]=\alpha(h) x \ 
\mathrm{for\ all}\ h\in \mathfrak{E}\ltimes \mathfrak{h} \}.\end{equation}
Define $\Pi=\{\alpha_i : i\in I\}\subseteq (\mathfrak{E}\ltimes\lie h)^{*}$ by $\alpha_j((D_k,h_i))=\delta_{k,j}+a_{i,j}$. The elements of $\Pi$ are called the simple roots of $\lie g$. 
Set
$$Q:=\bigoplus _{\a\in \Pi}\mathbb{Z}\alpha,\ \ Q_+ :=\sum _{\a\in \Pi}\mathbb{Z}_{+}\alpha \ \  \text{and}\ \ Q_+^{\mathrm{im}} :=\sum _{\a\in \Pi_{\mathrm{im}}}\mathbb{Z}_{+}\alpha.$$
The coroot associated with $\a\in \Pi$ is denoted by $h_\a.$
The set of roots is denoted by $\Delta :=\{ \alpha \in (\mathfrak{E}\ltimes\lie h)^*\backslash \{0\} \mid \mathfrak{g}_{\alpha }\neq 0\}$ and
the set of positive roots is denoted by  $\Delta_+:=\Delta\cap Q_+$. 
The elements in $\Pi^\mathrm{re}:=\{\alpha_i: i\in I^{\mathrm{re}}\}$ and  $\Pi^\mathrm{im}:=\Pi\backslash \Pi^\mathrm{re}$  are called the set of real simple roots and
the set of imaginary simple roots. We have $\Delta =\Delta_+ \sqcup - \Delta_+$ and 
$$\mathfrak{g}_0=\mathfrak{h},\ \ \lie g_\alpha=\lie g(k_1, \dots,k_n),\ \text{ if }\ \alpha=\sum_{i\in I} k_i\alpha_i\in \Delta.$$
\noindent
Moreover, we have a triangular decomposition
$\lie g\cong \lie n^{-}\oplus \lie h \oplus \lie n^+,$
where
$\lie n^{\pm}=\bigoplus_{\alpha \in \pm\Delta_{+}}
\mathfrak{g}_{\alpha}.$
Given $\gamma=\sum_{i\in I}k_i\alpha_i\in Q_+$, we set $\text{ht}(\gamma)=\sum_{i\in I}k_i.$ 
Finally, for $\lambda, \mu\in (\mathfrak{E}\ltimes \mathfrak{h})^*$  we say that $\lambda\ge \mu$ if $\lambda-\mu\in Q_+.$

\subsection{The Weyl group}
The real vector space spanned by $\Delta$ is denoted by $R=\mathbb{R}\otimes_{\bz} Q$. There is a symmetric bilinear form on $R$ given by $(\alpha_i, \alpha_j)=d_i a_{ij}$ 
for $i, j\in I.$ For $\alpha\in \Pi^{\mathrm{re}}$, define the linear isomorphism $\bold{s}_\alpha$ of $R$ by 
$$\bold{s}_\alpha(\lambda)=\lambda-2\frac{(\lambda,\alpha)}{(\alpha,\alpha)}\alpha,\ \ \lambda\in R.$$
The Weyl group $W$ 
of $\mathfrak{g}$ is the subgroup of $\mathrm{GL}(R)$ generated by the simple reflections $\bold{s}_\alpha$, $\alpha\in \Pi^\mathrm{re}$.
Note that the above bilinear form is $W$--invariant and $W$ is a Coxeter group with canonical generators $\bold{s}_\alpha, \a\in \Pi^\mathrm{re}$. 
Define the length of $w\in W$ by  $\ell(w):=\mathrm{min}\{k\in \mathbb{N}: w=\bold{s}_{\a_{i_1}}\cdots \bold{s}_{\a_{i_k}}\}$ and any expression $w=\bold{s}_{\a_{i_1}}\cdots \bold{s}_{\a_{i_k}}$ with 
$k=\ell(w)$ is called a reduced expression. The set of real roots is denoted by $\Delta^\mathrm{re}=W(\Pi^\mathrm{re})$ and the set of imaginary roots is denoted by
$\Delta^\mathrm{im}=\Delta\backslash \Delta^\mathrm{re}$. 
Equivalently, a root $\alpha$ is real if and only if $(\alpha, \alpha)> 0$ and else imaginary. We can extend $(.,.)$ 
to a symmetric form on $(\mathfrak{E}\ltimes \lie h)^*$ satisfying $(\lambda,\alpha_i)=\lambda(d_ih_i)$ and also
$\bold{s}_\alpha$ to a linear isomorphism of $(\mathfrak{E}\ltimes \lie h)^*$ by 
$$\bold{s}_\alpha(\lambda)=\lambda-2\frac{(\lambda,\alpha)}{(\alpha,\alpha)}\alpha,\ \ \lambda\in (\mathfrak{E}\ltimes \lie h)^*.$$
Note that $\lambda(h_\a)=2\frac{(\lambda,\alpha)}{(\alpha,\alpha)}$ for $\a\in \Pi$.
Let $\rho$ be any element of $(\mathfrak{E}\ltimes \lie h)^*$ satisfying $2(\rho,\alpha)=(\alpha,\alpha)$ for all $\a\in \Pi$.


\subsection{Characters}
Denote by $\mathcal{O}^{\rm{int}}$ the category consisting of integrable $\lie g-$modules from the category $\mathcal{O}$ of $\lie g$.
Let
$P_+=\{\lambda\in (\mathfrak{E}\ltimes \mathfrak{h})^*: \lambda(h_{\a})\in\mathbb{Z}_+, \ \a\in \Pi\}$ 
be the set of all dominant integral weights of $\lie g$ and let $L(\lambda)$ be the irreducible highest weight module of $\lie g$ associated to $\lambda\in P_+.$
Then it is well-known that there exists a bijective correspondence between the irreducible objects in $\mathcal{O}^{\rm{int}}$ and $\{L(\lambda): \lambda\in P_+\}$ and the 
category $\mathcal{O}^{\rm{int}}$ is semi-simple (see \cite{Bor88, kang}).
Given $\lambda \in P_+$, the module $L(\lambda)$ has a weight space decomposition $L(\lambda) = \bigoplus_{\mu \in \mathfrak{h}^*} L(\lambda)_\mu$. 
The formal character of $L(\lambda)$ is defined to be  $\text{ch} L(\lambda)  := \sum_{\mu\in \mathfrak{h}^*} \dim(L(\lambda)_\mu)\, e^\mu$.  
Let $\Omega(\lambda)$ be the set of all $\gamma\in Q_+$ such that $\gamma$ is the sum of mutually orthogonal distinct imaginary simple roots which are orthogonal to $\lambda$.
We define the normalized Weyl numerator by:
\begin{equation}\label{Ulambda}
U_\lambda:=  \sum_{(w, \gamma) \in W\times \Omega(\lambda) } (-1)^{\ell(w)+\mathrm{ht}(\gamma)} e^{w(\lambda+\rho -\gamma)-(\lambda + \rho)} .
 \end{equation}
\noindent
Note that $0\in \Omega(\lambda)$ and that an imaginary simple root $\alpha$ is in $\Omega(\lambda)$ if $(\lambda, \alpha)=0.$
The Weyl-Kac character formula gives:

\begin{equation}\label{WeylKac}
\text{ch} L(\lambda)e^{-\lambda}=  
\frac{\sum\limits_{(w, \gamma) \in W\times \Omega(\lambda) } (-1)^{\ell(w)+\mathrm{ht}(\gamma)} e^{w(\lambda+\rho -\gamma)-(\lambda + \rho)}}{\sum\limits_{(w, \gamma) \in W\times \Omega(0) } (-1)^{\ell(w)+\mathrm{ht}(\gamma)} e^{(w\rho-\rho -w\gamma)}}=\frac{U_\lambda}{U_0}
\end{equation}
\noindent
Denote by $\mathcal{A}=\mathbb{C}[[X_\a : \a\in \Pi]]$ the formal power series ring on the variables $X_\a=e^{-\a}$ corresponding to the simple roots of $\lie g$.
Since $e^\mu$ appears in $\mathrm{ch} L(\lambda)e^{-\lambda}$ with nonzero coefficient only when $\lambda\ge \mu$, it follows that
 $U_\lambda\in \mathcal{A}$ for all $\lambda\in P_+.$

\subsection{Special dominant weights}\label{specialdominant} A dominant weight $\lambda\in P_+$ is said to be special if $(\lambda, \a)=0$ for all $\a\in \Pi^{\rm{im}}.$
An integrable irreducible $\lie g$--module $L(\lambda)$ is called special if $\lambda$ is special. We record the following simple lemma for our future use.
\begin{lem}\label{speciallem}
Let $\lambda\in P_+.$ Then we have,
  $\lambda$ is special and $W$--invariant if and only if $L(\lambda)$ is one--dimensional if and only if  $\mathrm{ch} L(\lambda)=e^\lambda$.\qed
\end{lem}
Note that irreducible $\lie g$--module $L(\lambda)$ need not be one-dimensional if $\lambda$ is $W-$invariant. 

\subsection{}\label{defncg}
In this subsection, we recall some definitions and results about general simple graphs that were used to prove that unique factorization property in the Kac--Moody case 
(see \cite[Section 4.2]{VV12}).
Let $\mathcal{G}$ be a simple graph with vertex set $V$ consisting of $n$ elements. 
We say a subset $S$ of $V$ is {\em{independent}} if there is no edge between any two  elements of $S$, i.e. the subgraph spanned by $S$ in $\mathcal{G}$ is totally disconnected. 
The following definition is important for describing the factors of $U_\lambda$ (see also \cite[Section 4.2]{VV12}).
\begin{defn}
A $k$--partition of the graph $\mathcal{G}$ is an ordered $k$--tuple $(J_1,J_2,...,J_k)$ such that the following conditions hold
\begin{enumerate}
 \item the $J_i$ are non-empty pairwise disjoint subsets of the vertex set $V$ whose union is $V$; 
 \item each $J_i$ is an independent subset  of $V$.
\end{enumerate}
We denote by $P_k(\mathcal{G})$ the set of all $k$--partitions of $\mathcal{G}$ and $c_k(\mathcal{G}):=|P_k(\mathcal{G})|$.  We also define 
$$c(\mathcal{G}):=(-1)^n \sum\limits_{k=1}^{n}(-1)^k\frac{c_k(\mathcal{G})}{k}.$$
\end{defn}
\noindent
The following result describes the connection between $c(G)$ and the connectivity.
\begin{prop}\cite[Proposition 2]{VV12}\label{cG}
	For a graph $\mathcal{G}$, we have $c(\mathcal{G})\in \mathbb{Z}_+$ and $c(\mathcal{G})>0 $ if and only if $\mathcal{G}$ is connected.
\end{prop}

\subsection{Graph of $\lie g$ and $\lambda$}\label{defnconnected}
Let $G$ be the simple graph underlying the Dynkin diagram of $\lie g$, i.e. the vertex set of $G$ is $\Pi$ and two vertices $\alpha, \beta\in 
\Pi$ is connected by an edge if and only if $(\alpha, \beta)\neq 0$. We will refer to $G$ as the graph of $\lie g.$ 
For any subset $S\subseteq \Pi$, we denote $|S|$ by the number elements in $S$ and denote $G(S)$ by the subgraph spanned by the subset $S$. 
We say a subset $S\subseteq \Pi$ is connected if the corresponding subgraph spanned by $S$ is connected, i.e., $G(S)$ is connected.
For $\lambda \in P_+$, define $$\text{ $\lambda^{\perp}_{\mathrm{im}} = \{\alpha \in \Pi^{\mathrm{im}} : (\lambda, \alpha) = 0 \}$ and 
$\Pi(\lambda)=\Pi^{\mathrm{re}}\cup \lambda^{\perp}_{\mathrm{im}}$.} $$
Note that the set $\Pi(\lambda)$ need not be connected in general and $\Pi(\lambda)=\Pi$ if $\lambda$ is special. The subgraph spanned by $\Pi(\lambda)$ will be called the graph of $\lambda.$
Denote $\mathcal{C}(\lambda)$ by the set of all connected subsets of $\Pi(\lambda)$ and it is naturally a poset under the set inclusion.
The set of all connected components of $\Pi(\lambda)$ is denoted as
$\mathcal{MC}(\lambda)$ and let $k_\lambda=|\mathcal{MC}(\lambda)|$ denote the number of connected components of $\Pi(\lambda)$. 
Note that the elements of $\mathcal{MC}(\lambda)$ are nothing but the maximal elements in $\mathcal{C}(\lambda)$ with respect to set inclusion.
More generally for $\overline{\lambda}=(\lambda_1,\cdots, \lambda_k) \in P_+^{k}$, we denote $\mathcal{MC}(\overline{\lambda})$ to be the multiset 
$\sqcup_{i=1}^k \mathcal{MC}(\lambda_i)$.

\vskip 2mm
For $k\in \mathbb{N}$, we denote by $S_k$ the set of permutations of $\{1, \cdots, k\}$.


\section{Logarithm Expansion of Characters}\label{section3}
In this section, we introduce the elements $U(\lambda,\chi)$ and $U_i(\lambda,\chi)$ that generalize the normalized Weyl numerator $U_\lambda$.
We give all the arithmetic tools to prove criteria for their factorization, namely logarithm of characters and projection operators.
\subsection{The Weyl group action} We discuss the Weyl group action on the elements $\lambda+\rho-\gamma, \lambda\in P_+$ and $\gamma\in \Omega(\lambda)$  in order to understand the monomials 
$e^{w(\lambda+\rho-\gamma)-(\lambda+\rho)}.$
For $w\in W$,
we fix a reduced word expression $w=\bold {s}_{i_1}\cdots \bold{s}_{i_k}$ and let $I(w)=\{\alpha_{i_1},\dots,\alpha_{i_k}\}\subseteq \Pi^{\mathrm{re}}$. 
Note that $I(w)$ is independent of the choice of the reduced expression of $w$ (see \cite{Hump}). Let $\lambda\in P_+$ and for $\gamma\in \Omega(\lambda)$ we set 
$$I(\gamma)=\{\alpha\in \Pi^{\mathrm im} : \mbox{ $\alpha$ is a summand of $\gamma$}\},$$ namely $I(\gamma)=\{\a_{i_1}, \cdots, \a_{i_r}\}$ if
$\gamma=\sum_{k=1}^r\a_{i_k}$. Note that the sets $I(\gamma)$ depends on $\lambda$, but we simply denote it as $I(\gamma)$ for convenience. Define
$$\mathcal{J}(\lambda):=\left\{ (w, \gamma)\in W\times \Omega(\lambda)\backslash \{(1, 0)\}: I(w)\cup I(\gamma)\ \text{is an independent subset of $\Pi$}\right\}.$$
 Note that $(w, 0)\in \mathcal{J}(\lambda)$ if and only if $I(w)$ is an independent subset of $\Pi^{\mathrm{re}}$ and $(1, \gamma)\in \mathcal{J}(\lambda)$ for all nonzero $\gamma\in \Omega(\lambda)$. 
 We record the following simple lemma for future use (see \cite{Hump}).
\begin{lem}\label{lengthofw}
 If $I(w)$ is an independent subset of $\Pi^{\mathrm{re}}$ for $w\in W$, then we have $\ell(w)=|I(w)|.$
\end{lem}

Given $\lambda\in P_+$, $w\in W$ and $\gamma\in \Omega(\lambda)$, we define $b_{\alpha}^\lambda(w,\gamma)\in \mathbb Z$ for every $\alpha\in\Pi$ such that
$(\lambda + \rho)-w(\lambda+\rho-\gamma)=\sum_{\alpha\in\Pi} b_{\alpha}^\lambda(w,\gamma)\alpha$.
\begin{lem}\label{helplem}
For a given $\lambda\in P_+$, $w\in W$ and $\gamma\in \Omega(\lambda)$, 
we have:
\begin{enumerate}
\item[\emph{(i)}] $b_{\alpha}^\lambda(w,\gamma)\in \mathbb{Z}_{+}$ for all $\alpha\in \Pi$ and $b_{\alpha}^\lambda(w,\gamma)= 0$ if $\alpha\notin I(w)\cup I(\gamma)$;
\item[\emph{(ii)}] $I(w)=\big\{\alpha\in \Pi^{\mathrm re} : b_{\alpha}^\lambda(w,\gamma)\geq \frac{2(\lambda+\rho,\ \alpha)}{(\alpha,\ \alpha)}\big\}$;
\item[\emph{(iii)}] If $(w, \gamma)\in \mathcal{J}(\lambda)$, then $b_{\alpha}^\lambda(w,\gamma)=\frac{2(\lambda+\rho,\ \alpha)}{(\alpha,\ \alpha)}$ for all 
$\alpha\in I(w),\  b_{\alpha}^\lambda(w,\gamma)=1$ 
if $\alpha\in I(\gamma)$, and $b_{\alpha}^\lambda(w,\gamma)=0$, otherwise;
\item[\emph{(iv)}] If $(w, \gamma)\notin \mathcal{J}(\lambda)\cup \{(1, 0)\}$, then $\exists\ \alpha\in I(w)\subseteq  \Pi^{\mathrm re}$ such that
$b_{\alpha}^\lambda(w,\gamma)>\frac{2(\lambda+\rho,\ \alpha)}{(\alpha,\ \alpha)}$. 
\end{enumerate}

\begin{pf}
We prove the lemma by induction on $\ell(w)$. If $\ell(w)=0$, then (i)--(iii) are obvious. 
Let $\alpha\in \Pi^\mathrm{re}$ such that $w=\bold s_{\alpha}u$ and $\ell(w)=\ell(u)+1$. Then
\begin{align*}\notag (\lambda + \rho)-w(\lambda+\rho-\gamma)&=(\lambda+\rho)-\bold s_{\alpha}u(\lambda+\rho-\gamma)&\\&\label{1}=
(\lambda+\rho)-u(\lambda+\rho-\gamma)+\left(2\frac{(\lambda,u^{-1}\alpha)}{(\alpha,\alpha)}+2\frac{(\rho,u^{-1}\alpha)}{(\alpha,\alpha)}
-2\frac{(\gamma,u^{-1}\alpha)}{(\alpha,\alpha)}\right)\alpha.
\end{align*}
By induction hypothesis we know that $(\lambda+\rho)-u(\lambda+\rho-\gamma)$ has the required property, and since $\ell(w)=\ell(u)+1$ and $\a\in \Pi^{\mathrm{re}}$, we also know
$u^{-1}\alpha\in \Delta^{\mathrm re} \cap \Delta_+$. This implies that $2\frac{(\lambda+\rho,\ u^{-1}\alpha)}{(\alpha,\ \alpha)}\in \mathbb{N}$, because $\lambda\in P_+$.
Further, $\gamma$ is a sum of imaginary simple roots and $a_{ij}\leq 0$ whenever $i\neq j$. 
Hence $-2\frac{(\gamma,u^{-1}\alpha)}{(\alpha,\alpha)}\in \mathbb{Z}_{+}$ and since $I(w)=I(u)\cup \{\alpha\}$, the proof of (i) is done.
To prove the Statement (ii), observe that there are two possibilities: $I(w)=I(u)$ or $I(w)=I(u)\sqcup \{\a\}$. 
Statement (ii) is immediate if $I(w)=I(u)$. Suppose $I(w)=I(u)\sqcup \{\a\}$, then we have
 $u^{-1}\alpha-\a\in Q_+$ since $u^{-1}\alpha\in \Delta^{\mathrm re} \cap \Delta_+$ and  $\a\notin I(u)$. This implies that 
 $2\frac{(\lambda+\rho,\ u^{-1}\alpha-\a)}{(\alpha,\ \alpha)}\ge 0$ as $\lambda\in P_+$.
 Hence (ii) follows.
\vskip 1.5mm
If $(w, \gamma)\in \mathcal{J}(\lambda)$, then we have $I(w)=I(u)\sqcup \{\a\}$ and $(u, \gamma)\in \mathcal{J}(\lambda)$. This implies that $u^{-1}\alpha=\alpha$ and so
$(\rho,u^{-1}\alpha)=(\rho,\alpha)=\frac{1}{2}(\alpha,\alpha)$ and $(\gamma,u^{-1}\alpha)=(\gamma,\alpha)=0$. 
Thus (iii) follows from the above expression of $(\lambda+\rho)-w(\lambda+\rho-\gamma)$  by induction.

\vskip 1.5mm
It remains to prove (iv). Suppose $w=\bold s_{\alpha}$ we have
$$(\lambda + \rho)-w(\lambda+\rho-\gamma)=\gamma+\left(1+2\frac{(\lambda,\alpha)}{(\alpha,\alpha)}-2\frac{(\gamma,\alpha)}{(\alpha,\alpha)}\right)\alpha.$$
Since $(w, \gamma)\notin \mathcal{J}(\lambda)\cup \{(1, 0)\}$, $\{\a\}$ and $I(\gamma)$ are not independent and
we get $-2\frac{(\gamma,\alpha)}{(\alpha,\alpha)}\in\mathbb{N}$.  This completes the proof of (iv) for $w=\bold s_\a$, giving us a base for induction.
For the induction step we write $w=\bold s_{\alpha}u$ where $\ell(w)=\ell(u)+1$. We have either 
$(u, \gamma)\in\mathcal{J}(\lambda)\cup \{(1, 0)\}$ or 
$(u, \gamma)\notin\mathcal{J}(\lambda)\cup \{(1, 0)\}$. 
In the latter case, we are done by using the induction hypothesis since $I(u)\subseteq I(w)$.
So, assume that $(u, \gamma)\in\mathcal{J}(\lambda)\cup \{(1, 0)\}$. Since $(w, \gamma)\notin \mathcal{J}(\lambda)\cup \{(1, 0)\}$, we have $(u, \gamma)\neq (1, 0)$ and
$I(w)=I(u)\sqcup \{\a\}$ and there are two possibilities: 
either $I(u)\cup \{\a\}$ is independent or $\{\a\}\cup I(\gamma)$ is independent. However we can not have that both of them are independent.
\vskip 1.5mm
\noindent
Case 1: Suppose $I(u)\cup \{\a\}$ is independent and $\{\a\}\cup I(\gamma)$ is not independent,  then we have $(\gamma,u^{-1}\alpha)=(\gamma,\alpha)<0$ 
and $-2\frac{(\gamma,u^{-1}\alpha)}{(\alpha,\alpha)}\in \mathbb{N}$, so we are done by interpreting this in the above expression of $(\lambda+\rho)-w(\lambda+\rho-\gamma)$.

\noindent
Case 2. Suppose $\{\a\}\cup I(\gamma)$ is independent and $I(u)\cup \{\a\}$ is not independent, then we have $u^{-1}(\a)-\a\in Q_+\backslash \{0\}$ since 
$u^{-1}\alpha\in \Delta^{\mathrm re} \cap \Delta_+$ and $\a\notin I(u).$
So, we get $2\frac{(\lambda+\rho,\ u^{-1}\alpha-\a)}{(\alpha,\ \alpha)}\ge 1$. 
Now interpreting this in the above expression of $(\lambda+\rho)-w(\lambda+\rho-\gamma)$, the assertion follows.
\end{pf}
\end{lem}

\begin{rem} The above lemma is  a generalization of \cite[Lemma 2]{VV12} in the setting of Borcherds--Kac--Moody algebras, see also \cite[Lemma 3.6]{AKV17} for the special case $\lambda=0$. 
\end{rem}
\subsection{The elements $U(\lambda, \chi)$}\label{elmUlambchi}
We define the elements $U(\lambda,\chi)$ in $\mathcal A$ which generalize the Weyl denominators $U_\lambda$. 
We shall later determine in which cases these elements admit unique factorization and how
they factorize in other cases. 

Let  $\chi: W\times Q^{\mathrm{im}}_+ \to \mathbb C \setminus \{0\}$ be a homomorphism that is,
 $$\chi(1, 0)=1 \ \text{and} \ \chi\left(\prod\limits_{i=1}^{r}w_i, \sum\limits_{i=1}^{r}\gamma_i\right)=\chi(w_1, \gamma_1)\cdots \chi(w_r, \gamma_r)$$ for $w_1,\ldots, w_r\in W$ and $\gamma_1,\ldots, \gamma_r\in Q^{\mathrm{im}}_+.$
For $\lambda \in P_+$ and such $\chi$, we define the normalized Weyl numerator associated with $(\lambda, \chi)$ as follows:
$$U(\lambda, \chi):= \sum_{(w,\gamma) \in W \times \Omega(\lambda) } \chi(w, \gamma)e^{w(\lambda+\rho-\gamma)-(\lambda + \rho)}$$
Since $w(\lambda+\rho)\le \lambda+\rho$ for all $w\in W$, we have $U(\lambda, \chi)\in \mathcal{A}$. 
 There are two important homomorphisms which are especially useful. 
One is the trivial homomorphism defined to be 
$$\mathds{1}(w, \gamma):=1$$ for all $(w, \gamma)\in W\times Q^{\mathrm{im}}_+ \label{one morphism}$ and another one is the 
 sign homomorphism which is defined by 
 $$\sign(w, \gamma):=(-1)^{\ell(w)}(-1)^{\mathrm{ht}(\gamma)} \label{sign morphism}$$
Clearly, $U(\lambda, \sign)=U_\lambda$, that is $U(\lambda, \chi)$ generalizes the Weyl numerator, and 
 $$U(\lambda, \mathds{1})= \sum_{(w,\gamma) \in W \times \Omega(\lambda) } e^{w(\lambda+\rho-\gamma)-(\lambda + \rho)}.$$

\subsection{Logarithm of $U(\lambda, \chi)$}
We recall the notion of logarithm of elements in $\mathcal A$ which is our main tool to make the product $U(\lambda_1,\chi)\cdots U(\lambda_r,\chi)$ into a sum $\sum_{i=1}^r\log U(\lambda_i,\chi)$ which is less difficult to analyze.
Given an element $f\in \mathcal{A}$, we define
$$-\mathrm{log}(1-f)=f+\frac{f^2}{2}+\frac{f^3}{3}+\cdots+\frac{f^k}{k}+\cdots .$$
Note that $-\mathrm{log}(1-f)\in \mathcal{A}$ for all $f\in \mathcal{A}$ and $\mathrm{log}(fg)=\mathrm{log}(f)+\mathrm{log}(g)$ for all 
elements $f, g\in \mathcal{A}$ with constant terms $1$.

\vskip 1.5mm
\noindent
Now write $\zeta(\lambda, \chi):=1-U(\lambda, \chi)$ and for $(w, \gamma)\in W\times \Omega(\lambda)$ we denote the monomial 
$$
X(\lambda, w,\gamma):=e^{w(\lambda+\rho -\gamma)-(\lambda + \rho)}=\prod\limits_{\alpha\in\Pi}X_{\alpha}^{b_{\alpha}^\lambda(w,\gamma)}.$$ 
For a subset $C$ of $\Pi$, denote by 
$$
X^\lambda(C)  :=
 \prod\limits_{\alpha \in C^{\mathrm{re}}}X_{\alpha}^{\langle\lambda+\rho,\ \alpha^{\vee}\rangle} \prod\limits_{\alpha \in C^{\mathrm{im}}}X_{\alpha}$$ and 
$$
\chi(C)  
=\prod\limits_{i\in C^{\mathrm{re}}}\chi (\bold s_i,0)\prod\limits_{i\in C^{\mathrm{im}}}\chi(e, \a_i).$$

One of the key ingredients in the proof of our main theorem is to we investigate the elements $-\mathrm{log}\ U(\lambda, \chi)$ 
and give necessary and sufficient conditions for the monomials $X^\lambda(C)$ to appear in $-\mathrm{log}\ U(\lambda, \chi)$ with nonzero coefficient. 
The following proposition is a generalization of \cite[Proposition 1]{VV12} for Borcherds--Kac--Moody algebras.
Recall the definitions of $\Pi(\lambda)=\Pi^{\rm{re}}\cup \lambda^{\perp}_{\rm{im}}$ and $c(\mathcal{G})$ from the Sections \ref{defnconnected} and \ref{defncg}.
\begin{prop}\label{keyproposition-main}
Let $\chi: W\times Q^{\mathrm{im}}_+ \to \mathbb C \setminus \{0\}$ be a homomorphism and let $\lambda\in P_+$ and $C\subseteq \Pi$.
\begin{enumerate}
\item[\emph{(i)}] Suppose the monomial $\prod_{\a\in\Pi}X_\a^{a_\a}$ appears in $-\mathrm{log}\ U(\lambda, \chi)$ with nonzero coefficient then the support of this monomial
 must be contained in $\Pi(\lambda),$ i.e. $$\mathrm{supp}\big(\prod_{\a\in\Pi}X_\a^{a_\a}\big)=\{\a\in \Pi : a_\a\neq 0 \}\subseteq \Pi(\lambda).$$

\item[\emph{(ii)}]  The coefficient of the monomial $X^\lambda(C)$ in $-\mathrm{log}\ U(\lambda, \chi)$ is equal to $\chi(C)c(G(C))$ for $C\subseteq{\Pi(\lambda)}$, where
$G(C)$ is the subgraph spanned by $C.$ In particular, this coefficient depends only on the subset $C.$

 \item[\emph{(iii)}] The monomial $X^\lambda(C)$ appears in $-\mathrm{log}\ U(\lambda, \chi)$ with nonzero coefficient if and only if $C$ is a connected subset of $\Pi(\lambda)$.

\end{enumerate}

\end{prop}
\begin{pf} Part 
(iii) of the proposition follows from  (ii) and the Proposition \ref{cG}, since $\chi(C)\neq 0$ for all $C\subseteq \Pi$. Now
write $$\zeta(\lambda, \chi) = -\sum_{\substack{(w,\gamma) \in W \times \Omega(\lambda) \\ (w, \gamma) \ne (1, 0)} } \chi(w, \gamma) X(\lambda, w,\gamma)  = \zeta_1(\lambda, \chi) + \zeta_2(\lambda, \chi)$$ where 
$$\zeta_1(\lambda, \chi) = -\sum\limits_{\substack{(w,\gamma) \in W \times \Omega(\lambda) \\ (w, \gamma) \in \mathcal{J}(\lambda)}} \chi(w, \gamma) X(\lambda, w,\gamma) $$ and 
$$\zeta_2(\lambda, \chi) = -\sum\limits_{\substack{(w,\gamma) \in W \times \Omega(\lambda) \\ (w, \gamma) \notin \mathcal{J}(\lambda)\cup \{(1, 0)\}}}  \chi(w, \gamma)  X(\lambda, w,\gamma) $$
Since $-\mathrm{log}\ U(\lambda, \chi) = \zeta(\lambda, \chi) + \frac{\zeta(\lambda, \chi)^2}{2} + \cdots + \frac{\zeta(\lambda, \chi)^k}{k}+\cdots$, 
 Lemma \ref{helplem} implies that any monomial 
 \vskip 1.5mm
 \noindent
 $\prod_{\alpha \in \Pi^{\mathrm{re}}}X_{\alpha}^{p_{\alpha}} \prod_{\alpha \in \Pi^{\mathrm{im}}} X_{\alpha}^{m_{\alpha}} $ 
that occur in $-\mathrm{log}\ U(\lambda, \chi)$ with nonzero coefficient must satisfy the
 \vskip 1.5mm
 \noindent
following conditions: 
$$ \text{$(1)$\ if $p_\a\neq 0$ for some  $\alpha \in \Pi^{\mathrm{re}}$ then  $p_{\alpha} \ge  (\lambda+\rho)(h_\a)$ and
$(2)$ $m_\a=0$ for $\a\notin \Pi(\lambda).$}$$
\noindent
This proves the statement (i) of the proposition.
In particular, the monomial $X^\lambda(C)$ appears in $-\mathrm{log}\ U(\lambda, \chi)$ with nonzero coefficient then $C\subseteq{\Pi(\lambda)}$.
\vspace{0.5mm}
\noindent
Let $C\subseteq{\Pi(\lambda)}$, then Lemma \ref{helplem} (iv) further implies that there is no contribution of $\zeta_2(\lambda, \chi)$ 
to the coefficient of the monomial $X^\lambda(C)$ in 
$\zeta(\lambda, \chi) + \frac{\zeta(\lambda, \chi)^2}{2}+ \cdots + \frac{\zeta(\lambda, \chi)^k}{k}+\cdots $, i.e.
\small{
\begin{align*}\notag\text{the coefficient of}\ X^\lambda(C) \ \text{in} \  \left(\sum\limits_{k=1}^{\infty} \frac{\zeta(\lambda, \chi)^k}{k}\right) &= 
\ \text{the coefficient of}\ X^\lambda(C)\ \text{in} \ 
\left(\sum\limits_{k=1}^{\infty} \frac{\zeta_1(\lambda, \chi)^k}{k}\right)&\\&= \sum\limits_{k=1}^{\infty} \left(\text{the coefficient of}\ X^\lambda(C)\ \text{in} \ \frac{\zeta_1(\lambda, \chi)^k}{k}\right).
\end{align*}}
\noindent
Hence it is enough to calculate the coefficient of $X^\lambda(C)$ in $\zeta_1(\lambda, \chi)^k,$  where  
$$\zeta_1(\lambda, \chi)^{k} = \sum\limits_{\substack{((w_1,\gamma_1),(w_2,\gamma_2),...,(w_k,\gamma_k)) \\ (w_i,\gamma_i) ~\in~ W \times \Omega(\lambda) 
\\ (w_i, \gamma_i) \in \mathcal{J}(\lambda)}} (-1)^k \prod\limits_{i=1}^{k}\chi(w_i, \gamma_i)
X(\lambda, w_i,\gamma_i).$$ 
\noindent
From Lemma \ref{helplem} (iii), we get that $ \prod\limits_{i=1}^{k}X(\lambda, w_i,\gamma_i)=X^\lambda(C)$ if and only if
$(I(w_1)\cup I(\gamma_1), I(w_2)\cup I(\gamma_2),...,I(w_k)\cup I(\gamma_k))$ form a $k$--partition of $C$. In particular, for this $k$--partition
of $C$ we have 
$$\prod\limits_{i=1}^{k}\chi(w_i, \gamma_i)=\chi\left(\prod\limits_{i=1}^{k}w_i, \sum\limits_{i=1}^{k}\gamma_i\right)=
\prod\limits_{i\in C^{\mathrm{re}}}\chi (\bold s_i,0)\prod\limits_{i\in C^{\mathrm{im}}}\chi(e, \a_i)=\chi(C)$$ 
which is independent of the choice of the partition and
depends only upon the subset $C$ and the homomorphism $\chi$. 
Since all the $k$--partitions of $C$ arise in this way and $\chi$ is a homomorphism, using Lemma \ref{lengthofw}, we have
$$\text{\big(the coefficient of $X^\lambda(C)$ in $\zeta_1(\lambda, \gamma)^k\big)=(-1)^k \chi(C) |P_k(G(C))|$} $$ where $G(C)$ is the subgraph spanned by $C$.
Putting this all together, we conclude that $$\big(\text{the coefficient of $X^\lambda(C)$ in 
$-\mathrm{log}\ U(\lambda, \chi)\big)=\sum\limits_{k=1}^{\infty}\frac{(-1)^k}{k}\chi(C)|P_k(G(C)|=\chi(C)c(G(C)).$}$$
Since $\chi(C)\neq 0$ for all $C\subseteq \Pi$, we have $\chi(C)c(G(C))$ is nonzero if and only if $c(G(C))$ is nonzero. Thus Proposition \ref{cG} completes the proof of (ii).

\end{pf}
\subsection{Factorization of $U(\lambda,\chi)$}
The set $\Pi(\lambda)$ need not be connected in general for $\lambda\in P_+.$
We factor $U(\lambda,\chi)$ to a product of elements $U_i(\lambda,\chi)$ which correspond to the connected components  of the graph corresponding to $\Pi(\lambda)$. 

Let $\lambda\in P_+$  and 
let $C_i^\lambda,\  1\le i\le k_\lambda,$ be the connected components of $\Pi(\lambda)$. 
For $1\le i\le k_\lambda$, we denote by $W(\lambda)_i$ the Weyl group generated by the simple reflections $\{\bold s_\a : \a\in C_i^\lambda\cap \Pi^{\mathrm{re}}\}$. 
Similarly, for $1\le i\le k_\lambda$, we denote by $\Omega(\lambda)_i$ 
 the set of all $\gamma\in Q_+$ such that $\gamma$ is the sum of mutually orthogonal distinct imaginary simple roots which are in
$C_i^\lambda\cap \lambda^\perp_{\mathrm{im}} $. Then we have
\begin{equation}\label{Ulambdaprod}
U(\lambda, \chi)= \prod\limits_{i=1}^{k_\lambda}\left(
\sum\limits_{(w, \gamma) \in W(\lambda)_i\times \Omega(\lambda)_i} \chi(w, \gamma) e^{w(\lambda+\rho-\gamma)-(\lambda+\rho)}\right)
 \end{equation}
 For $1\le i\le k_\lambda$, we denote
 \begin{equation}\label{Ulambda_i}
 U_i(\lambda, \chi):=\sum\limits_{(w, \gamma) \in W(\lambda)_i\times \Omega(\lambda)_i} \chi(w, \gamma) e^{w(\lambda+\rho-\gamma)-(\lambda+\rho)}\end{equation}
 Then Equation (\ref{Ulambdaprod}) becomes $$U(\lambda, \chi)=U_1(\lambda, \chi)\cdots U_{k_\lambda}(\lambda, \chi).$$ 

In the following proposition, we make some observations about the monomials which appear in $-\mathrm{log}\ U_i(\lambda, \chi)$, which are similar to
the observations about $U(\lambda, \chi)$ that were stated in Proposition \ref{keyproposition-main}.
 \begin{prop}\label{keyproposition-main-ithcomponent}
 Let $\chi: W\times Q^{\mathrm{im}}_+ \to \mathbb C \setminus \{0\}$ be a homomorphism and let $\lambda\in P_+$.
Let $C_1^\lambda,\cdots, C_{k_\lambda}^\lambda$ be the connected components of $\Pi(\lambda)$. For $1\le i\le k_\lambda$, we have:
\begin{enumerate}
\item[\emph{(i)}] The support of a monomial which appears in $-\mathrm{log}\ U_i(\lambda, \chi)$ with nonzero coefficient  must be contained in $C_i^\lambda$.

\item[\emph{(ii)}] The coefficient of the monomial $X^\lambda(C)$ in $-\mathrm{log}\ U_i(\lambda, \chi)$ is equal to $\chi(C)c(G(C))$ for any $C\subseteq{C_i^\lambda}$, 
and in particular this coefficient depends only on the subset $C$.

 \item[\emph{(iii)}] For a subset $C\subseteq C_i^\lambda$, 
 the monomial $X^\lambda(C)$ appears in $-\mathrm{log}\ U_i(\lambda, \chi)$ with nonzero coefficient if and only if $C$ is a connected subset of $C_i^\lambda$.
 
\end{enumerate}

\end{prop}
\begin{proof}
 We leave out the proof since it closely follows the arguments of Proposition \ref{keyproposition-main}. 
\end{proof}

\subsection{The projection operator $\#_C$}
We define the projection operator $\#_C$ corresponding to a connected component $C$ of the graph of $\lambda$. We apply this operator on $ -\mathrm{log}\ U_i(\lambda, \chi) $ and use it to determine when $U_i(\lambda, \chi)$ is equal to $U_j(\mu, \chi)$.

Recall that $\mathcal{A}=\mathbb{C}[[X_\a : \a\in \Pi]]$ where $X_\a=e^{-\a}$. Let $\lambda\in P_+$ and $C$ be a connected component of $\Pi(\lambda).$
Define the map $\#_C:\mathcal{A}\to \mathcal{A}$ which maps
$$\text{$f=\sum\limits_{\bold a} X^{a_\a}_\a\mapsto f^{\#_C}=\sum\limits_{\substack{\bold a \\ \mathrm{supp}(\bold a)=C}} X^{a_\a}_\a,$}$$
where $\bold a=(a_\a : \a\in \Pi)$ and $\mathrm{supp}(\bold a)=\{\a\in \Pi: a_\a\neq 0\}.$ It is easy to see that the operator $\#_C$ is a linear operator.
The following proposition follows from Proposition \ref{keyproposition-main-ithcomponent} and it plays a crucial role in the proof of our main theorem.
\begin{prop}\label{logkeyproposition}
Let $\chi: W\times Q^{\mathrm{im}}_+ \to \mathbb C \setminus \{0\}$ be a homomorphism.
 Let $\lambda\in P_+$ and  $C=C_i^\lambda$ be a connected component of $\Pi(\lambda).$
 Then we have $$(-\mathrm{log}\ U_i(\lambda, \chi))^{\#_C}= \chi(C) c(G(C))X^\lambda(C)+ \ 
 \text{monomials of degree $> \mathrm{deg}\ X^\lambda(C)$}.$$ where 
 $\mathrm{deg}\ X^\lambda(C)=\sum_{\a\in C^{\mathrm{re}}}\lambda(h_\a)+|C^{\mathrm{im}}|$.\qed
\end{prop}

The following lemma compares when two $U_i(\lambda, \chi), U_j(\mu, \chi)$ are equal.
\begin{lem}\label{keylemmasulamb}
Let $\chi: W\times Q^{\mathrm{im}}_+ \to \mathbb C \setminus \{0\}$ be a homomorphism.
Let $\lambda, \mu\in P_+$ and let $C_i^\lambda$ and $C_j^{\mu}$ be two connected components $\Pi(\lambda)$ and $\Pi(\mu)$ respectively. 
Then the following statements are equivalent:

\begin{enumerate}
 \item[\emph{(i)}] $X^\lambda(C_i^\lambda)=X^{\mu}(C_j^\mu)$
 \item[\emph{(ii)}] $C_i^\lambda=C_j^\mu$ and $\lambda(h_\a) = \mu(h_\a)$ for all $\a\in C_i^\lambda\cap \Pi^{\mathrm{re}}$
 \item[\emph{(iii)}] $U_i(\lambda, \chi)=U_j(\mu, \chi)$
\end{enumerate}

\end{lem}
\begin{proof}
 If  $X^\lambda(C_i^\lambda)=X^{\mu}(C_j^\mu)$ then the supports and exponents of the corresponding variables of these monomials must be equal, 
 so (ii) follows from (i).
 Since $C_i^\lambda=C_j^\mu$  we get 
 $C_i^\lambda\cap\Pi^{\mathrm{re}}=C_j^\mu\cap\Pi^{\mathrm{re}}$ and $C_i^\lambda\cap\Pi^{\mathrm{im}}=C_j^\mu\cap\Pi^{\mathrm{im}}$.
 This immediately implies $W(\lambda)_i=W(\mu)_j, \ \text{and}\ \Omega(\lambda)_i=\Omega(\mu)_j$.
 Since  $\lambda(h_\a)=\mu(h_\a)$ for all $\alpha\in C_i^\lambda\cap\Pi^{\mathrm{re}}$,
 we have $w(\lambda+\rho)-(\lambda+\rho)=w(\mu+\rho)-(\mu+\rho)$ for all $w\in W(\lambda)_i$. Hence we have $U_i(\lambda, \chi)=U_j(\lambda, \chi)$.
 Finally, the fact that (iii) implies (i) follows from Proposition \ref{logkeyproposition}.

\end{proof}


\section{Unique factorization properties}\label{main theorem section}
\subsection{} In this section, we prove that  the product $U(\lambda_1, \chi) \cdots U(\lambda_r, \chi)$ factorizes uniquely to terms corresponding to the connected components of 
 $\Pi(\lambda_1),\ldots,\Pi(\lambda_r)$. 
  More precisely, let $r, s\in \mathbb{N}$ and $\overline{\lambda}=(\lambda_1,\ldots, \lambda_r)\in P_+^r, \ \ \overline{\mu}=(\mu_1,\ldots, \mu_s)\in P_+^s$ such that
 \vskip 1mm
\begin{equation}\label{Ulambsec}
  U(\lambda_1, \chi) \cdots U(\lambda_r, \chi) = U(\mu_1, \chi) \cdots U(\mu_s, \chi) \end{equation}                                                                                        
Recall that, we can write  $U(\lambda_p, \chi)=\prod\limits_{i=1}^{k_{\lambda_p}} U_i(\lambda_p, \chi)$ for $1\le p\le r$ (see Equation (\ref{Ulambda_i}))   and similarly,
$U(\mu_q, \chi)=\prod\limits_{j=1}^{k_{\mu_q}} U_j(\mu_q, \chi)$ for $1\le q\le s$. 
Then Equation (\ref{Ulambsec}) becomes
\begin{equation}\label{factorizationsec}
\prod\limits_{p=1}^{r}\left(\prod\limits_{i=1}^{k_{\lambda_p}} U_i(\lambda_p, \chi) \right)=\prod\limits_{q=1}^{s}
\left(\prod\limits_{j=1}^{k_{\mu_q}} U_j(\mu_q, \chi)\right)
\end{equation}
We now prove that the factors occur in the equation \ref{factorizationsec} are unique up to permutation. i.e., number of factors on the both sides are equal 
\big($N=\sum\limits_{p=1}^{r}k_{\lambda_p}=\sum\limits_{q=1}^{s}k_{\mu_q}$\big) and there exists an permutation $\sigma$ on $S_N$ such that
$U_i(\lambda_p, \chi)=U_{i'}(\mu_{p'}, \chi)$ where $\sigma(i, p)=(i', p')$.
Precisely, 
\begin{thm}\label{mainUlambdachi}
Let $\chi: W\times Q^{\mathrm{im}}_+ \to \mathbb C \setminus \{0\}$ be a homomorphism.
Let $r, s\in\mathbb{N}$ and let $\overline{\lambda}=(\lambda_1,\ldots, \lambda_r)\in P_+^r, \ \ \overline{\mu}=(\mu_1,\ldots, \mu_s)\in P_+^s$. Then the following statements are equivalent:
\begin{enumerate}
 \item   $U(\lambda_1, \chi) \cdots U(\lambda_r, \chi) = U(\mu_1, \chi) \cdots U(\mu_s, \chi)$ 
 \vskip 2mm
  \item there exists a bijection $\sigma: \mathcal{MC}(\overline{\lambda})\to \mathcal{MC}(\overline{\mu})$
 given by $C\mapsto C_\sigma$ satisfying the following conditions: if $C \in \mathcal{MC}(\lambda_i)$ maps to $C_\sigma\in\mathcal{MC}(\mu_j)$ then
  \vskip 1mm
\begin{enumerate}
 \item[\emph{(i)}]  $C=C_\sigma$ 
 and further, \emph{(ii)}  $\lambda_i(h_\a) = \mu_j(h_\a)$ for all $\alpha \in C^{\mathrm{re}}$.
\end{enumerate}
  \vskip 1mm
\item $\sum\limits_{p=1}^{r}k_{\lambda_p}=\sum\limits_{q=1}^{s}k_{\mu_q}=: N$ and there exists a permutation $\sigma\in S_N$ such that
$$U_i(\lambda_p, \chi)=U_{i'}(\mu_{p'}, \chi)\ \text{where $\sigma(i, p)=(i', p')$.}$$
\end{enumerate}
\end{thm}
\begin{pf}
Suppose
\begin{equation}\label{Ulamb}
  U(\lambda_1, \chi) \cdots U(\lambda_r, \chi) = U(\mu_1, \chi) \cdots U(\mu_s, \chi). \end{equation}
Factorize this further in terms of $U_i(\lambda, \chi)$ as before, then we get
  \begin{equation}\label{factorization}
\prod\limits_{p=1}^{r}\left(\prod\limits_{i=1}^{k_{\lambda_p}} U_i(\lambda_p, \chi) \right)=\prod\limits_{q=1}^{s}
\left(\prod\limits_{j=1}^{k_{\mu_q}} U_j(\mu_q, \chi)\right)
\end{equation}
Applying $-\mathrm{log}$ on both sides we get, 
\begin{equation}\label{log}
 \sum\limits_{p=1}^{r}\sum\limits_{i=1}^{k_{\lambda_p}}-\mathrm{log}\ U_i(\lambda_p, \chi) =  
 \sum\limits_{q=1}^{s}\sum\limits_{j=1}^{k_{\mu_q}}-\mathrm{log}\ U_j(\mu_q, \chi)\end{equation}

Let $\mathcal{C}=\left\{C^{\lambda_p}_i: p\in \{1,\cdots, r\},\ \ 1\le i\le k_{\lambda_p}\right\} \cup \left\{C^{\mu_q}_j:  q\in \{1,\cdots, s\}, \ \ 1\le j\le k_{\mu_q}\right\}$
to be the set of all connected components of $\Pi(\lambda_i), \Pi(\mu_j), 1\le i\le k_{\lambda_p}, 1\le j\le k_{\mu_q}.$
We fix a maximal element $C$ from $\mathcal{C}$ with respect to set inclusion. 
Without loss generality we assume that this chosen maximal element occurs on the left hand side of the Equation (\ref{log}), say $C=C^{\lambda_1}_{i_0}$, 
and satisfies the following property: 
the monomial $X^\lambda(C)$ is the minimal degree monomial (with respect to total degree) among all the monomials in the left hand side of the Equation (\ref{log}) with $C$ as 
support, where $\lambda_1=\lambda$.
Now apply the map $\#_C$ to Equation (\ref{log}) then we get,
\begin{equation}\label{logg}
 \sum\limits_{p=1}^{r}\sum\limits_{i=1}^{k_{\lambda_p}}-\mathrm{log}\ U_i(\lambda_p, \chi)^{\#_C} =  
 \sum\limits_{q=1}^{s}\sum\limits_{j=1}^{k_{\mu_q}}-\mathrm{log}\ U_j(\mu_q, \chi)^{\#_C}.\end{equation}
Since $C$ is maximal in $\mathcal{C}$ with respect to set inclusion, we immediately observe the following from Proposition \ref{keyproposition-main-ithcomponent}: 
\begin{itemize}
 \item[\emph{(i)}] either $-\mathrm{log}\ U_i(\lambda_p, \chi)^{\#_C}$ 
 is equal to zero 
 or 

\item[\emph{(ii)}]  we have $C^{\lambda_p}_i=C$ and 
  the monomial $X^{\lambda_p}(C)$ is the minimal degree  monomial in 
  $-\mathrm{log}\ U_i(\lambda_p, \chi)^{\#_C}$ with respect to the total degree when $-\mathrm{log}\ U_i(\lambda_p, \chi)^{\#_C}$ is nonzero.
\end{itemize}
Similarly we have:
\begin{itemize}
 \item[\emph{(i)}] either $-\mathrm{log}\ U_j(\mu_q, \chi)^{\#_C}$  is equal to zero 
 or 
\item[\emph{(ii)}]  we have $C^{\mu_q}_j=C$ and 
  the monomial $X^{\mu_q}(C)$ is the minimal degree  monomial in $-\mathrm{log}\ U_j(\mu_q, \chi)^{\#_C}$ with respect to the total degree when 
  $-\mathrm{log}\ U_j(\mu_q, \chi)^{\#_C}$ is nonzero.
\end{itemize}
Since $X^{\lambda}(C)$ appears in the left hand side of Equation (\ref{logg}) with nonzero coefficient and it is of minimal degree monomial, 
there must exists 
$1\le q\le s$ and $1\le j\le k_{\mu_q}$ such that $C=C^{\mu_q}_j$ 
and $X^\lambda(C)=X^{\mu_q}(C^{\mu_q}_j)$ which immediately implies that
$U_{i_0}(\lambda, \chi)=U_j(\mu_q, \chi)$ from  Lemma \ref{keylemmasulamb}. 
Canceling $U_{i_0}(\lambda, \chi)$ and $U_j(\mu_q, \chi)$ in Equation (\ref{Ulamb}) and proceeding by induction we get the desired the result.
Converse part is clear from Lemma \ref{keylemmasulamb}.
\end{pf}

We end this section with the following important corollary of Theorem \ref{mainUlambdachi}.
If we assume all the sets $\Pi(\lambda)$ involved in Theorem \ref{mainUlambdachi} are connected then we get the unique factorization property for $U(\lambda, \chi)$.
\begin{cor}\label{maincorconnected}
Let $\chi: W\times Q^{\mathrm{im}}_+ \to \mathbb C \setminus \{0\}$ be a homomorphism.
Let $r, s\in\mathbb{N}$ and let $\overline{\lambda}=(\lambda_1,\ldots, \lambda_r)\in P_+^r, \ \ \overline{\mu}=(\mu_1,\ldots, \mu_s)\in P_+^s$ such that
 $\Pi(\lambda_i)$ and $\Pi(\mu_j)$ are connected for all $1\le i\le r,\  1\le j\le s$. Suppose we have, 
\begin{equation}\label{Ulambthm}
  U(\lambda_1, \chi) \cdots U(\lambda_r, \chi) = U(\mu_1, \chi) \cdots U(\mu_s, \chi) \end{equation} then  $r=s$ and                                                                                         
there exists a bijection $\sigma$ in $S_r$ such that $U(\lambda_i, \chi)=U(\mu_{\sigma(i)}, \chi)$ for all $1\le i\le r.$
\end{cor}
\begin{pf}
Since $\Pi(\lambda)$ are connected, we have $\mathcal{MC}(\overline{\lambda})=\{\Pi(\lambda_i):1\le i\le r\}$ and $\mathcal{MC}(\overline{\mu})=\{\Pi(\mu_j):1\le j\le s\}$.
Interpreting this in Theorem \ref{mainUlambdachi} gives us the result. \end{pf}

\begin{rem}
 If we specialize Theorem \ref{mainUlambdachi} and its Corollary \ref{maincorconnected}
 to the trivial homomorphism $\chi=\mathds{1}$ then we get the unique factorization properties of $U(\lambda, \mathds{1})$s as in 
 Theorem \ref{mainUlambdachi} and Corollary \ref{maincorconnected}.
 These unique factorization properties of $U(\lambda, \mathds{1})$ are new as far as we know.
\end{rem}

\subsection{ Unique Factorization of Tensor Products.}\label{mainthmuniquefacttensor}
By specializing Theorem \ref{mainUlambdachi} to the sign homomorphism $\chi=\rm{sgn}$, 
we can determine the necessary
and sufficient conditions for which the tensor product of irreducible representations from $\mathcal{O}^{\rm{int}}$
is isomorphic to another.
\begin{thm}\label{tensormainthm}
Let $\lie g$ be a Borcherds--Kac--Moody algebra and
let $r, s\in\mathbb{N}$ such that $r\ge s$ and let $\overline{\lambda}=(\lambda_1,\ldots, \lambda_r)\in P_+^r, \ \ \overline{\mu}=(\mu_1,\ldots, \mu_s)\in P_+^s$.
Set $\widetilde{\mu}=(\mu_1,\ldots, \mu_s, 0,\cdots, 0)$ where $0$ appears $r-s$ times.
We have,\begin{equation}\label{irrthmmain}
 L(\lambda_1) \otimes L(\lambda_2) \otimes \cdots \otimes L(\lambda_r) \cong L(\mu_1) \otimes L(\mu_2) \otimes \cdots \otimes L(\mu_s)
\end{equation}
as $\lie g$--modules if and only if $\sum\limits_{i=1}^{r}\lambda_i=\sum\limits_{j=1}^{s}\mu_j$ and
there exists a bijection $\sigma: \mathcal{MC}(\overline{\lambda})\to \mathcal{MC}(\widetilde{\mu})$
 given by $C\mapsto C_\sigma$ satisfying the following conditions: if $C \in \mathcal{MC}(\lambda_i)$ maps to $C_\sigma\in\mathcal{MC}(\mu_j)$ then we have
\begin{enumerate}
 \item[\emph{(i)}]  $C=C_\sigma$ 
 and further, \ \ \emph{(ii)} $\lambda_i(h_{\a}) = \mu_j(h_{\a})$ for all $\alpha\in C^{\mathrm{re}}$.
\end{enumerate}
\end{thm}

\begin{pf}
 By taking character on both sides of Equation (\ref{irrthmmain}), we get
\begin{equation}\label{equalchar}
 \mathrm{ch}L(\lambda_1) \cdots \mathrm{ch}L(\lambda_r)=\mathrm{ch}L(\mu_1) \cdots  \mathrm{ch}L(\mu_s)
\end{equation}
By considering the maximal weights on the both sides of Equation (\ref{irrthmmain}) we get $\sum\limits_{i=1}^{r}\lambda_i=\sum\limits_{j=1}^{s}\mu_j=:\tau$ (say).
Now multiply $e^{-\tau}$ on the both side of the equation \ref{equalchar} and group the elements $\mathrm{ch}L(\lambda)e^{-\lambda}$, then we get 
$$\prod\limits_{i=1}^{r}\mathrm{ch}L(\lambda_i)e^{-\lambda_i}= \prod\limits_{j=1}^{s}\mathrm{ch}L(\mu_j)e^{-\mu_j}.$$ Using the Weyl-Kac character formula (i.e., the equation \ref{WeylKac}), we get 
$$ U_{\lambda_1} \cdots U_{\lambda_r}= U_{\mu_1} \cdots U_{\mu_s}U_0^{r-s}.$$
Now the necessary condition is immediate from Theorem \ref{mainUlambdachi} by specializing for $\chi=\rm{sgn}$.
For the converse part, it is immediate from Theorem \ref{mainUlambdachi} (again specialize to $\chi=\rm{sgn}$)  that $ U_{\lambda_1} \cdots U_{\lambda_r}= U_{\mu_1} \cdots U_{\mu_s}U_0^{r-s}.$
Now using $\sum\limits_{i=1}^{r}\lambda_i=\sum\limits_{j=1}^{s}\mu_j$,  we get 
$e^{\sum\limits_{i=1}^{r}\lambda_i} U_{\lambda_1} \cdots U_{\lambda_r}=e^{\sum\limits_{j=1}^{s}\mu_j} U_{\mu_1} \cdots U_{\mu_s}U_0^{r-s}.$ Thus we get,
$\prod\limits_{i=1}^{r}\mathrm{ch}L(\lambda_i)= \prod\limits_{j=1}^{s}\mathrm{ch}L(\mu_j).$
This immediately implies isomorphism of the corresponding tensor products \ref{irrthmmain} since $\mathcal{O}^{\mathrm{int}}$ is completely reducible.
 \end{pf}
 
We can immediately deduce the following corollary which gives the unique factorization property for tensor products of special family  of irreducible integrable representations of 
indecomposable Borcherds--Kac--Moody algebra $\lie g$.
\begin{cor}\label{uniquecor}
Let $\lie g$ be an indecomposable Borcherds--Kac--Moody algebra and
let $r, s\in\mathbb{N}$ and let $\overline{\lambda}=(\lambda_1,\ldots, \lambda_r)\in P_+^r, \ \ \overline{\mu}=(\mu_1,\ldots, \mu_s)\in P_+^s$
such that $\Pi(\lambda_i)$ and $\Pi(\mu_j)$ are connected for all $1\le i\le r,\  1\le j\le s$.
Suppose \begin{equation}\label{irrthm}
 L(\lambda_1) \otimes L(\lambda_2) \otimes \cdots \otimes L(\lambda_r) \cong L(\mu_1) \otimes L(\mu_2) \otimes \cdots \otimes L(\mu_s)
\end{equation}
as $\lie g$--modules then $r=s$ and there exists a permutation $\sigma\in S_r$ such that
 $$\text{$\mathrm{ch} L(\lambda_i)=e^{\lambda_i-\mu_{\sigma(i)}}\mathrm{ch} L(\mu_{\sigma(i)})$ for all $1\le i\le r.$}$$
 Further more if we assume that $\lambda_i-\mu_{\sigma(i)}$ are special for all $1\le i\le r$ then there exists one--dimensional $\lie g-$modules
 $Z_i$ such that $L(\lambda_i)\cong L(\mu_{\sigma(i)})\otimes Z_i$ for all $1\le i\le r.$
\end{cor}
\begin{pf}
Since $\Pi(\lambda_i)$ and $\Pi(\mu_j)$ are connected for all $1\le i\le r,\  1\le j\le s$, by Theorem \ref{tensormainthm}, we get $r=s$ and there exists a permutation
 $\sigma\in S_r$ such that $\Pi(\lambda_i)=\Pi(\mu_j)$ and $\lambda_i(h_\a)=\mu_{\sigma(i)}(h_\a)$ for all $\a\in \Pi^{\rm{re}}$. 
 This implies that $\lambda_i-\mu_{\sigma(i)}$ is $W-$invariant and we get
 $\mathrm{ch} L(\lambda_i)=e^{\lambda_i-\mu_{\sigma(i)}}\mathrm{ch} L(\mu_{\sigma(i)})$ for all $1\le i\le r$
 since $\Omega(\lambda_i)=\Omega({\mu_{\sigma(i)}})$ for all $1\le i\le r$. The last statement immediately follows from Lemma $\ref{speciallem}$.
\end{pf}

\begin{rem}
Suppose that the Borcherds--Cartan matrix is indecomposable and $a_{ii}=2$ for all $i\in I$. 
Then we have $\Pi=\Pi_{re}$ and $\Pi_{\rm{im}}=\emptyset$. Since $\Pi(\lambda)=\Pi$ is connected for all $\lambda\in P_+$, we recover the result for Kac--Moody algebras proved in \cite{VV12}.
This result was first proved for finite dimensional simple Lie algebras in \cite{Rajan}.
\end{rem}

\subsection{Examples}\label{examplesuniquefails}
We will end with few examples which helps us to understand our results.
\begin{example}
Suppose that the Borcherds--Cartan matrix is $\left[\begin{array}{ccc}
2 & -1 & 0\\
-1 & 0 & -1\\
0 & -1 & 2
\end{array}\right]$. Denote by $\Pi_{re}=\{\alpha_1,\alpha_2\}$ and $\Pi_{im}={\beta}$. 
Note that $(\alpha_1,\alpha_2)=0$, $(\a_1, \beta)=(\a_2, \beta)=-1$ and $(\beta, \beta)=0$.
In particular, the Weyl group $W=W_1\times W_2$ where $W_i=\{1, \bold s_{\a_i}\}$ for $i=1, 2$. 
Let $\Gamma=\{\lambda\in P_+: (\lambda, \beta)=0\}$. 
For $\lambda\in \Gamma,$ we have $\beta\in \lambda^{\perp}_{\rm{im}}$ and in particular $\lambda$ is a special dominant weight such that $\Pi(\lambda)=\Pi$ is connected. 
So, Corollary \ref{uniquecor} applies
to elements of $\Gamma$. Suppose $\mu\in P_+\backslash \Gamma$, then we have $\Pi(\mu)=\{\a_1, \a_2\}$ which is not connected. Write
$\mu=a_1\omega_1+a_2\omega_2+a_3\omega_3$ with $a_3\neq 0$, where $\omega_i, i\in I$ are fundamental weights defined by $\omega_i(h_j)=\delta_{ij}$ for all $i, j\in I$.
Since $\Omega(\mu)=0$, we have 
$$\text{$\ch L(\mu)
=e^{(a_3+1)\omega_3}\frac{U^1_{a_1\omega_1}\cdot U^2_{a_2\omega_2}}{e^{\rho}U_0}$
where $U^i_{a_i\omega_i}=\sum\limits_{w\in W_i}(-1)^{\ell(w)}e^{w(a_i+1)\omega_i}$, $i=1, 2$.}$$
Thus for $a_1,a_2,b_1,b_2\in \mathbb{Z}_+$ and $a_3, b_3\in \mathbb{N}$, we get
$$L(a_1\omega_1+a_2\omega_2+a_3\omega_3)\otimes L(b_1\omega_1+b_2\omega_2+b_3\omega_3)=
L(a_1\omega_1+b_2\omega_2+a_3\omega_3)\otimes L(b_1\omega_1+a_2\omega_2+b_3\omega_3).$$
\end{example}

\begin{rem}
A similar phenomena happens for typical finite-dimensional modules over Lie superalgebras. For example, take $\mathfrak{g}=\mathfrak{sl}(m|n)$
 and $\Pi$ to be the standard choice of simple roots. Then for $\lambda$ typical, $\mathcal{MC}(\lambda)$ has two components and the character of 
 $L(\lambda)$ factorize to two terms. However, due to the lack of complete reducibility, this is not enough to conclude about tensor products of representations. 
 The subcategory of typical finite-dimensional $\mathfrak{g}$-modules admits complete reducibility but is not closed under tensor products. \end{rem}

\begin{example}
 Suppose that $a_{ij}\in -\mathbb{Z}_+$ for all $i, j.$ Then we have $\Pi^{\rm{re}}=\emptyset$ and $\Pi=\Pi^{\rm{im}}$.
 Let $\lambda\in P_+$ and let $k_i\in \mathbb{N}, 1\le i\le r$ and $\ell_i\in \mathbb{N}, 1\le i\le r$ such that $\sum\limits_{i=1}^{r}k_i=\sum\limits_{j=1}^{r}\ell_j=:s.$
 Then we have $\ch (L(k_1\lambda)\otimes\cdots \otimes L(k_r\lambda))=e^{(s-r)\lambda}\ch L(\lambda)=\ch (L(\ell_1\lambda)\otimes\cdots \otimes L(\ell_r\lambda))$ and hence we have
 $L(k_1\lambda)\otimes\cdots \otimes L(k_r\lambda)\cong  L(\ell_1\lambda)\otimes\cdots \otimes L(\ell_r\lambda).$
\end{example}





\bibliographystyle{plain}

\end{document}